\documentclass[12pt]{article}
\usepackage[left=2cm, right=2cm, top=2cm, bottom=2cm]{geometry}

\usepackage[all]{xy}
\usepackage[affil-it]{authblk}
\usepackage{amsmath} 
\usepackage{amssymb} 
\usepackage{makeidx} 
\usepackage{graphicx} 
\usepackage{enumerate}
\usepackage[utf8]{inputenc}
\usepackage{amsthm}
\usepackage{color}
\usepackage{cancel}
\usepackage{multirow}

\theoremstyle{plain}
\newtheorem{theorem}{Theorem}[section]
\newtheorem{lemma}[theorem]{Lemma}

\theoremstyle{definition}

\newtheorem{remark}[theorem]{Remark}

\newtheorem{definition}[theorem]{Definition}

\numberwithin{equation}{section}
\newtheorem{example}[theorem]{Example}

\long\def\symbolfootnote[#1]#2{\begingroup%
\def\thefootnote{\fnsymbol{footnote}}\footnote[#1]{#2}\endgroup}

\def\00{{\bf 0}}
\def\11{{\bf 1}}
\def\+{\oplus}

\newcommand{\boxtensor}{{\Box\kern-9.03pt\raise1.42pt\hbox{$\times$}}}

\newcommand{\be}{\begin{eqnarray}}
\newcommand{\ee}{\end{eqnarray}}

\usepackage{color}
\def\meidl#1 {\fbox {\footnote {\ }}\ \footnotetext { From Wilfried: {\color{red}#1}}}
\def\hmeidl#1 {}

\begin{document}

\title{On nilpotent automorphism groups of function fields}

\author{Nurdag\"{u}l Anbar  and  Bur\c{c}in G\"{u}ne\c{s}
\vspace{0.4cm} \\
\small{Sabanc{\i} University}\\
\small MDBF, Orhanl\i, Tuzla, 34956 \. Istanbul, Turkey\\
\small Email: {\tt nurdagulanbar2@gmail.com}\\
\small Email: {\tt bgunes@sabanciuniv.edu }
 }

\date{}

\maketitle

\begin{abstract}
We study the automorphisms of a function field of genus $g\geq 2$ over an algebraically closed field of characteristic $p>0$. More precisely, we show that the order of a nilpotent subgroup $G$ of its automorphism group is bounded by $16 (g-1)$ when G is not a $p$-group. We show that if $|G|=16(g-1) $, then $g-1$ is a power of $2$. Furthermore, we provide an infinite family of function fields attaining the bound.
\end{abstract}


\section{Introduction}\label{introduction}


Let $ K $ be an algebraically closed field and $F$ be a function field of genus $g$ with constant field~$K$. 
Denote by $\mathrm{Aut}(F/K) $ the automorphism group  of $ F $ over $ K $.
It is a well-known fact that if $ F$ is of genus $0$ or $1$, then $\mathrm{Aut}(F/K)$ is an infinite group.
However, this group is finite if $g \geq 2$, which is proved by Hurwitz \cite{Hurwitz1893} for $K=\mathbb{C}$ and 
by Schmid \cite{schmid1938} for $K$ of positive characteristic. 
In his paper, Hurwitz also showed that $|\mathrm{Aut}(F/K)| \leq 84(g-1)$, which is now called the Hurwitz bound.
This bound is sharp, i.e., there exist function fields of characteristic zero of arbitrarily high genus whose automorphism group has order $ 84(g - 1) $, see \cite{Macbeath}. Roquette \cite{roquette1970} showed that the Hurwitz bound also holds in positive characteristic $p$, if $p$ does not divide $|\mathrm{Aut}(F/K) |$. 
We remark that the Hurwitz bound does not hold in general. 
In positive characteristic, the best known bound  is  
\begin{center}
	$ |\mathrm{Aut}(F/K) | \leq 16g^4 $
\end{center}
with one exception: the Hermitian function field. 
This result is due to Stichtenoth \cite{Sti1, Sti2}.

\noindent However, there are better bounds for the order of certain subgroups of automorphism groups. Let $ G \leq \mathrm{Aut}(F/K)$.
Nakajima \cite{Nakajima} showed that if $G$ is abelian, then $ |G| \leq 4(g+1)$ for any characteristic.
Furthermore, Zomorrodian \cite{Zom} proved that 
$$
|G| \leq 16(g-1),
$$ 
when $K=\mathbb{C}$ and $G$ is a nilpotent subgroup.
He also showed that if the equality holds, then $g-1$ is a power of $2$.
Conversely, if $g-1$ is a power of $2$, then there exists a function field of genus $g $ that admits an automorphism group of order $ 16(g-1) $; whence a nilpotent group of order power of $ 2 $. We remark that his approach is based on the method of Fuchsian groups.

In this paper, we give a similar bound for the order of any nilpotent subgroup of the automorphism group of a function field in positive characteristic except for one case. 
More precisely, our main result is as follows.
\begin{theorem}\label{thm:main}
	Let $K$ be an algebraically closed field of characteristic $p>0$ and $F/K$ be a function field of genus $g\geq 2$. 
	Suppose that $G \leq  \mathrm{Aut}(F/K)$ is a nilpotent subgroup of order $$|G|>16(g-1). $$ 
	Then the following holds.
	\begin{itemize}
		\item[(i)] $ G $ is a $p$-group.
		\item[(ii)] The fixed field $ F_0 $ of $ G $ is rational. 
		\item[(iii)] There exists a unique place $ P_0 $ of $ F_0 $, which is ramified in $ F/F_0 $. Moreover, $ P_0 $ is totally ramified, and 
		$$ 
		|G| \leq \frac{4p}{(p-1)^2}g^2.
		$$
	\end{itemize}	
\end{theorem}
%

\begin{remark}
In the exceptional case in Theorem~\ref{thm:main}, since there is a unique ramified place $ F $ has $p$-rank zero by \cite[Corollary~2.2]{AST}.
\end{remark}
\begin{remark}
We conclude from the proof of Theorem~\ref{thm:main} that the bound $ |G| \leq 16(g-1) $ also holds when $ p=0 $. Moreover, when the bound is attained the order of $ G $ is a power of $ 2 $. 
\end{remark}

\section{Preliminary Results}\label{pre}

In this section, we recall some basic notions related to function fields and give some preliminary results, which will be our main tools for the proof of Theorem~\ref{thm:main}. For more details about function fields, we refer to \cite{HKT, thebook}.

 Let $K$ be an algebraically closed field of characteristic $ p $ and $F\supseteq E$ be a finite separable extension of function fields over $K$ of genus $g(F)$ and $g(E)$, respectively. Denote by $\mathbb{P}_F$ the set of places of $F$. For a place $Q\in \mathbb{P}_F$ lying above $P\in \mathbb{P}_E$, we write $Q|P$, and denote by $e(Q|P)$ the ramification index and by $d(Q|P)$ the different exponent of $Q|P$.  Recall that $Q|P$ is ramified if $e(Q|P)>1$.
Moreover, if $p$ does not divide $e(Q|P)$, then it is called tamely ramified; otherwise it is called wildly ramified. By Dedekind's Different Theorem \cite[Theorem 3.5.1]{thebook}, $Q|P$ is ramified if and only if $d(Q|P)>0$. 
More precisely, $d(Q|P)\geq e(Q|P)-1$ and the equality holds if and only if $Q|P$ is tame. Note that every place of $F$ has degree $1$ as $K$ is algebraically closed; hence, the genera of $F$ and $E$ are related by the Hurwitz Genus Formula \cite[Theorem 3.4.13]{thebook} as follows.
\begin{equation}\label{eq:hgf1}
2g(F)-2 = [F:E](2g(E)-2) + \sum_{\substack{Q\in \mathbb{P}_F, P\in \mathbb{P}_E\\
Q|P}} d(Q|P),
\end{equation}
where $[F:E]$ is the extension degree of $F/E$. 

\noindent From now on, we suppose that $F/E$ is a Galois extension with Galois group $G$. As $F/E$ is Galois, $[F:E]=|G|$. Let $Q_1, \ldots, Q_m \in \mathbb{P}_F$ be all extensions of $P\in \mathbb{P}_E$. Since $G$ acts transitively on $Q_1, \ldots, Q_m$,  we have
$$ e(Q_i|P)=e(Q_j|P)=:e(P) \quad \text{and} \quad d(Q_i|P)=d(Q_j|P)=:d(P) $$
for all $i,j\in \lbrace 1,\ldots,m \rbrace$. Then by the Fundamental Equality \cite[Theorem 3.1.11]{thebook}, 
Equation~\eqref{eq:hgf1} can be written as 
\begin{align}\label{eq:hgf}
2g(F)-2=|G|\left( 2g(E)-2+ \sum_{ P\in \mathbb{P}_E }\frac{d(P)}{e(P)}\right).
\end{align}
Equation~\eqref{eq:hgf} and the following well-known lemma will be our main tools to give an upper bound for the order of a nilpotent subgroup of the automorphism group of a function field. 
\begin{lemma}
If 	$ G $ is a finite nilpotent group, then $G$ has a normal subgroup of order $n$ for each divisor $n$ of $ |G| $.
\end{lemma}
\begin{proof}
The proof follows from the fact that every finite nilpotent group is a direct product of its Sylow subgroups, see \cite{Hungerford}.
\end{proof}

Let $G$ be a subgroup of $\mathrm{Aut}(F/K)$. 
We denote the fixed field $F^G$ of $G$ by
$F_0$ and the genus of $F_0$ by $g_0$. 
Note that $F/F_0$ is a Galois extension with Galois group $ G$. Set $N:= |G|$.

\begin{lemma}\label{l1}
	If $g_0 \geq 1$, then $N \leq 4(g-1)$. 
\end{lemma}
\begin{proof}
	If $g_0 \geq 2$, then by Equation~\eqref{eq:hgf}, we conclude that $2g-2 \geq 2 N$, i.e., $ N \leq (g-1) $. 
	If $ g_0=1$, then 
	$$ 
	\displaystyle 2g-2= N\left( \sum_{\substack{P\in \mathbb{P}_{F_0}}}\frac{d(P)}{e(P)}\right),
	$$ 
	by Equation~\eqref{eq:hgf}. Since $g \geq 2 $, there exists a ramified place $P_0 \in \mathbb{P}_{F_0}$. 
	Hence,
	$$ 
	2g-2 \geq N  \frac{d(P_0)}{e(P_0)} \geq N  \frac{(e(P_0)-1)}{e(P_0)} .
	$$
	This implies that $ N \leq 4(g-1) $ as $e(P_0) \geq 2$.
\end{proof}
\noindent From now on, we assume that $G$ is a nilpotent subgroup of $\mathrm{Aut}(F/K)$. By Lemma~\ref{l1}, we also assume  that $g(F_0)=0$. 

\begin{definition}
	Suppose that  $P_1, \ldots, P_r$ are all places of ${F_0}$, which are ramified in $F$ with ramification indices $e_1, \ldots, e_r $ and different exponents $d_1, \ldots, d_r $, respectively. We can without loss of generality assume that $e_1 \leq \ldots \leq e_r $. In this case, we say that $F/F_0$ (or shortly $F$) is of type $ (e_1,e_2, \ldots, e_r) $.
\end{definition}

\begin{lemma} \label{lemma1}
	Let $\ell$ be a prime number. Then $ \ell|N $ if and only if  $\ell|e_i$ for some $i~\in~\{1, \ldots, r \} $.
\end{lemma}
\begin{proof} 
	If $ \ell|e_i $ for some $i \in \{1, \ldots, r \} $, then $ \ell|N $ since $ e_i |N $.
	Suppose that $ \ell |N $ and $\ell \nmid e_i$ for any $i = 1, \ldots, r$.
	Since $G$ is nilpotent, there is a normal subgroup $H $ of $ G$ such that $ [G:H]=\ell $. Set $ F_1:=F^H $. 
	Note that $ F_1 / F_0 $ is an unramified Galois extension of degree $\ell$. 
	Then by Equation~\eqref{eq:hgf} and the assumption $g(F_0)=0$, we obtain that $2g(F_1) - 2 = -2\ell$. That is, $g(F_1) = -\ell+1<0$, which is a contradiction. 
\end{proof}

\begin{lemma} \label{lemma2}
If $\ell$ is a prime number, which divides exactly one of $e_1, \ldots, e_r$, then $\ell=p$.
\end{lemma}
\begin{proof}
	Let $H$ be a normal subgroup of $G$ of index $[G:H]=\ell  $ and  $ F_1=F^H $. 
	Then there is only one place of $F_0$, which is ramified in $F_1/F_0$, say $P_1$. Suppose that $P_1$ is tamely ramified; equivalently, $\ell\neq\mathrm{char}(K)$, which is $p$. Then by Equation~\eqref{eq:hgf} 
	$$
	2g(F_1)-2 = \ell\bigg(-2+\frac{d_1}{e_1}\bigg)= \ell\bigg(-2+\frac{\ell-1}{\ell}\bigg)=-\ell-1< -2,
	$$
	which is a contradiction.
\end{proof}

\begin{lemma}\label{lem:two}
Suppose that $\ell$ is a prime dividing $N$ and $\ell \neq \mathrm{char}(K)$. Say $N=\ell^aN_1$ for some integers $a, N_1 \geq 1$ such that $\gcd(\ell,N_1)=1$. Then we have the following.
\begin{itemize}
\item[(i)] There exist at least two places, whose ramification indices are divisible by $\ell$.
\item[(ii)] If there are exactly two places, whose ramification indices are divisible by $\ell$, then their ramification indices are divisible by $\ell^a$.
\end{itemize} 
\end{lemma}
\begin{proof}
By Lemma~\ref{lemma1}, we know that $\ell|e_i$ for some $i \in \{1, \ldots, r \} $. Then $(i)$ follows from 
Lemma~\ref{lemma2} as $\ell\neq \mathrm{char}(K)$. Suppose that there are exactly two places whose ramification indices are divisible by $\ell$. 
Let $H $ be a normal subgroup of $ G$ of index $[G:H]=\ell^a$ and $F_1=F^H$. We consider the Galois extension $F_1/F_0 $ of degree $\ell^a$. Note that there are exactly two ramified places of $F_0$ in $F_1$. Since $F_1/F_0 $ is a tame extension and $g(F_0)=0$, by Equation~\eqref{eq:hgf}, we conclude that they are totally ramified, which proves $(ii)$.
\end{proof}

\begin{lemma}\label{1wild} Let $p=\mathrm{char}(K)$ and $ |G| = p^aN_1 $ for some integers $a, N_1\geq 1$ such that $ \gcd(p,N_1)~=~1 $. Let $P$ be a wildly ramified place of $F_0$ in $F$ with ramification index  $e(P)=p^tn$ for some positive integers $t\leq a$ and $n$ such that $\gcd (p,n)=1$. Then we have the following.
	\begin{enumerate}[(i)]
	\item\label{different} $ d(P) \geq (e(P)-1) + n(p^t-1) $.
		\item\label{order} If $P$ is the unique wildly ramified place of $F_0$ in $F$, then $t=a$.
	\end{enumerate}
\end{lemma}
\begin{proof}  
Let $H $ be a normal subgroup of $ G$ of index $[G:H]=p^a$ and $F_1=F^H$. Then $ F_1/F_0 $ is a Galois $ p $-extension of degree $ p^a $.
	\begin{enumerate}[(i)]
	\item  
		Let $P'\in \mathbb{P}_{F_1}$ and $P''\in \mathbb{P}_{F}$ such that $P''| P'|P$.
		Then by the transitivity of the different exponent \cite[Corollary 3.4.12 ]{thebook},
		$$
		d(P)= e(P''|P')\cdot d(P'|P)+d(P''|P') = n\cdot d(P'|P) + (n-1).
		$$
		Also, $d(P'|P)\geq 2(p^t-1)$ by Hilbert's Different Formula \cite[Theorem 3.8.7]{thebook}; hence, 
		\begin{align*}
		d(P)\geq 2n(p^t-1)+(n-1)= (np^t -1) + n(p^t-1).
		\end{align*}
		Then the fact that $e(P)=np^t$ gives the desired result.
		
	\item By Lemma~\ref{lem:two}-$ (ii)$, we conclude that $ G $ is a $ p $-group. Suppose that $ P $ is not totally ramified in $ F$. Let $ P' $  be a place of $ F $ lying over $ P $. We denote by $ G_T(P'|P) $ the inertia group of $ P $. Note that since $P$ is not totally ramified $ G_T(P'|P) $ is a proper subgroup of $ G $. Then the fact that $ G $ is solvable implies that there exists a normal subgroup $ H $ such that $ G_T(P'|P) \leq H \leq G $ and $ [G:H]=p $. That is, $ F^H $ is a Galois extension of $ F_0 $ of degree $ p $ . Moreover, $ F^H /F_0$ is unramified as inertia group of a place of $ F $ lying over $ P $ lies in $ H $, which is a contradiction.

%
	\end{enumerate}
\end{proof}

\section{Proof of Theorem~\ref{thm:main}}
In this section, we prove Theorem~\ref{thm:main}. We first fix the following notation. We denote by
\begin{itemize}
	\item[] $F/K$ \hspace{3cm}a function field of genus $g\geq 2 $ over an algebraically closed field $ K $ \hspace*{4cm} of characteristic~$ p>0 $,
	\item[] $G \leq \mathrm{Aut}(F/K)$ \hspace{1cm} a nilpotent subgroup of $ \mathrm{Aut}(F/K) $,
	\item[] $N :=|G|>1$,
	\item[] $F_0 := F^G$ \hspace{2.1cm} the fixed field of $ G $.
\end{itemize}
Note that $ F/F_0 $ is Galois extension of degree $ [F:F_0]=N $. By Lemma~\ref{l1}, we can assume that $ g(F_0) = 0 $, that is, $ F_0 $ is rational. Suppose that $ F $ is of type $ (e_1,\ldots,e_r) $, where $ r $ is the number of places of $ F_0 $ that are ramified in $ F/F_0 $. Recall that  $e_1 \leq \ldots \leq e_r $. We will prove Theorem~\ref{thm:main} according to the number $ r $.
\begin{theorem}
	If $ r \geq 5 $, then $ N \leq 4(g-1) $.
\end{theorem}
\begin{proof}
	By Equation~\eqref{eq:hgf}, we have the following equalities.
	$$
	2g-2= N\bigg(-2+\sum_{i=1}^{r}\frac{d_i}{e_i}\bigg) \geq N\bigg(-2 + 5\cdot \frac{1}{2}\bigg)=\frac{N}{2},
	$$
	which gives the desired result.
\end{proof}

%

We consider $F$ of type $(e_1,e_2,e_3,e_4)$.
That is, there are exactly $4$ ramified places of $ F_0 $, say $ P_1,P_2,P_3,P_4 $, with ramification indices $ e_1,e_2,e_3,e_4 $ and different exponents $ d_1,d_2,d_3,d_4 $, respectively. Then we have the following result.
\begin{theorem}\label{thm:r=4}
If $r= 4$, then $N\leq 8 (g-1)$.
\end{theorem}
\begin{proof}
Note that if $e_2 \geq 3 $, then $ N\leq 4(g-1) $ since by Equation~\eqref{eq:hgf} 
$$
2g-2 = N\bigg(-2+\sum_{i=1}^{4}\frac{d_i}{e_i}\bigg) 
\geq N\bigg(-2 +\frac{1}{2} + 3\cdot \frac{2}{3}\bigg)=\frac{N}{2}.
$$
From now on, we suppose that $e_2 =2$, i.e., $e_1= e_2 =2 $. Similarly, by Equation~\eqref{eq:hgf}, if $ e_3 \geq 4 $, then $N \leq 4(g-1)$. Hence, we investigate $e_3\leq 3 $ into cases as follows.
\begin{itemize}
\item[(i)] $ e_3= 3 $:\\
If $e_4 \geq 6$, then 
$$
2g-2 \geq N \bigg(-2 + 2\cdot \frac{1}{2} + \frac{2}{3} + \frac{5}{6}\bigg)= \frac{N}{2}, 
$$
which implies that $N \leq 4 (g-1)$.

\noindent Note that $e_4\neq 5$ by Lemma~$\ref{lemma2}$, i.e., $F$ is either of type $ (2, 2, 3, 4) $ or of type  $ (2, 2, 3, 3) $. 

\noindent Assume that $F$ is of type $ (2, 2, 3, 4) $. Then $\mathrm{char}(K) = 3 $ by  Lemma~\ref{lemma2}; hence,
$$
2g-2 \geq N \bigg(-2 +2\cdot \frac{1}{2}+\frac{2\cdot (3-1)}{3}+\frac{3}{4}  \bigg)>  N, 
$$
i.e., $N<2(g-1)$.

\noindent Assume that $F$ is of type $(2, 2, 3, 3)$. Then  $ N = 2^a 3^b $ for some positive integers $ a$ and $b$ by Lemma~\ref{lemma1}. If $\mathrm{char}(K) = 2$ or $3$, then there are two wildly ramified places. Therefore, by 
Equation~\eqref{eq:hgf}, we obtain that $N \leq 2(g-1)$. Assume that $ \mathrm{char}(K) >3 $. By Lemma~\ref{lem:two}$-(ii)$, we conclude that $a=b=1$, i.e., $N=6$. Then by Equation~\eqref{eq:hgf}, we obtain that $g=2$; hence, $N= 6(g-1)$. 

\item[(ii)] $e_3 = 2$: \\
Write $ e_4 = 2^s m $ for some integers $s \geq 0$ and $m\geq 1$ such that $\gcd (2,m)=1$. That is, $F$ is of type $(2,2,2,2^s m)$. 

\noindent  If $m > 1$, then $m=\ell^t$ for a prime $\ell > 2$ and an integer $t\geq 1$. By Lemma~\ref{lemma2}, we conclude that $\ell= p$, where $p$ is $\mathrm{char}(K)$. Moreover, $N=2^ap^b$ for some integers $a$, $b$ such that 
$a\geq 1$ and $b\geq t$ by Lemma~\ref{lemma1}. As there is a unique wild ramification, $t=b$ by Lemma~\ref{1wild}$-(ii)$, i.e., $e_4=2^sp^b $, and $d_4\geq (2^sp^b-1)+2^s(p^b-1)$. Then by Equation~\eqref{eq:hgf}
\begin{align*}
	2g-2 &= N  \bigg( -2 + 3 \cdot \frac{1}{2} +\frac{(2^sp^b-1)+2^s(p^b-1)}{2^sp^b} \bigg) \geq N,
	\end{align*}
i.e., $N\leq 2(g-1)$.

\noindent If $m = 1$, then $F$ is of type $(2,2,2,2^s)$ and $N=2^a$. If $\mathrm{char}(K) =2$, then $N \leq g-1$. Suppose that $\mathrm{char} (K) > 2$. Then $ P_1, P_2,P_3, P_4 $ are all tamely ramified in $ F $; hence, by Equation~$ \eqref{eq:hgf} $
$$ 
	2g-2 = N \left(-2 + 3\cdot\frac{1}{2} + \frac{2^s-1}{2^s}\right).
	$$
	Note that $ s\geq 2 $ since $ g \geq 2 $, and  $N=\displaystyle\frac{2^{s+1}}{2^{s-1}-1}(g-1) \leq 8(g-1)$. 
\end{itemize}
\end{proof}

\begin{remark}
Note that if the bound $8(g-1)$ is attained by $F$, then $g-1$ is a power of $2$, $F$ is of type $(2,2,2,2^s)$ for some integer $s\geq 2$, and $ \mathrm{char} (K) \neq 2 $.
\end{remark}

%
Now we consider $F$ of type $(e_1,e_2,e_3)$.  
That is, there are exactly $3$ ramified places of $ F_0 $, say $ P_1,P_2,P_3 $, with ramification indices $ e_1,e_2,e_3 $ and different exponents $ d_1,d_2,d_3$, respectively. Then we have the following result.

\begin{theorem}\label{thm:r=3}
 If $r=3$, then $N\leq 16 (g-1)$.
\end{theorem}
\begin{proof}
If $e_1 \geq 4$, then $N\leq 8(g-1)$ by Equation~\eqref{eq:hgf}. Therefore, we investigate $e_1\leq 3 $ into cases as follows.
\begin{itemize}
\item[(i)] $e_1 = 3 $:\\
Note that if $e_2 \geq 5$, then $N < 8(g-1)$ by Equation~\eqref{eq:hgf}. 
\begin{itemize}
\item[(a)] $ e_2=4 $:\\
Then $F $ is of type $(3,4,e_3)$. By Lemma~\ref{lemma2}, the ramification index $ e_3 $ can have at most one prime divisor $ \ell >3 $. Then $ e_3 = 2^a3^b\ell^c $ for some integers $a,b,c \geq 0$ and $N= 2^s3^t\ell^c $ for some integers $s\geq 2 $ and
$t \geq  1 $. 

If $c > 0 $, then $ \ell=  p $ by Lemma~\ref{lemma2} and $ a= 2 $, $b=1 $ by Lemma~\ref{lem:two}, i.e., $ e_3 = 12p^c $. Also, by Lemma~\ref{1wild}, $d_3\geq (12p^c-1) + 12(p^c-1)$. Then by Equation~\eqref{eq:hgf}, we have $N< 2(g-1)$.

Assume that $c=0$. Note that $a=0$ is not possible in this case; otherwise $F$ would be of type $(3,4,3)$ by Lemma~\ref{lem:two}$-(ii)$. If $b=0$, then $\mathrm{char}(K)=3$ and $F$ is of type $(3,4,4)$. As $P_1$ is wildly ramified, $N < 4(g-1)$ by Equation~\eqref{eq:hgf}. If $a,b\neq 0$, i.e., $e_3\geq 6$, then $N \leq 8(g-1)$ by Equation~\eqref{eq:hgf}. 

\item[(b)] $ e_2=3 $:\\
By similar arguments, $F $ is of type $(3,3,e_3)$, where $e_3=3^a\ell^b$ for a prime $\ell\neq 3$ and integers $a,b~\geq~0$. Then by Lemma~\ref{lemma1} and Lemma~\ref{1wild}, $N= 3^s\ell^b $ for an integer $s\geq 1$. 

If $b > 0 $, then $\ell=p$. By Lemma~\ref{1wild}, $d_3\geq (3^ap^b-1) +3^a(p^b-1)$; hence, $N < 4(g-1)$ by 
Equation~\eqref{eq:hgf}. 

Now suppose that $b=0$. If $a=1$, then $\mathrm{char} (K)=3$; otherwise $ g=1 $. 
That is, all places are wildly ramified; hence, $N\leq (g-1)$ by Equation~\eqref{eq:hgf}.
If $a>1$, then $N \leq 9(g-1)$.
\end{itemize}

\item[(ii)] $e_1 = 2 $:\\
If $\mathrm{char}(K)=2$, i.e., $P_1$ is wildly ramified, then $ N \leq 6(g-1) $ by Equation~\eqref{eq:hgf}. From now on, we suppose that $\mathrm{char}(K) >2$. If $e_2 \geq 6$, then $N\leq 12(g-1)$ by Equation~\eqref{eq:hgf}. We investigate $e_2 \leq 5$ into cases as follows.
\begin{itemize}
\item[(a)] $e_2=5$:\\
Then $F $ is of type $(2,5,e_3)$, where $ e_3 = 2^a5^b\ell^c $ for a prime $ \ell \neq 2,5 $ and integers $a,b,c \geq 0$. As $\mathrm{char}(K) \neq 2 $, by Lemma~\ref{lem:two} and Lemma~\ref{1wild}, $ e_3 = 2\cdot5^b\ell^c $ and $N= 2\cdot5^t\ell^c $ for an integer $t \geq 1$. 

If $c > 0 $, then $\ell =p$ and $b=t=1$, i.e., $e_3 = 10p^c$.
Then $d_3 \geq (10p^c-1) +10(p^c-1)$ by Lemma~\ref{1wild}; hence, $N < 3(g-1)$ by Equation~\eqref{eq:hgf}.

\noindent If $c=0$, then $b\neq 0$, i.e., $e_3\geq 10$; hence, $N\leq 10(g-1)$ by Equation~\eqref{eq:hgf}.

\item[(b)] $e_2=4$:\\
Then $F $ is of type $(2,4,e_3)$, where $e_3=2^a\ell^b$ for a prime $ \ell > 2 $ and integers $a,b \geq 0$. 

If $b> 0$, then $\ell = p $ and $d_3=(2^ap^b-1)+2^a(p^b -1 )$; hence, $N < 3(g-1)$ by Equation~\eqref{eq:hgf}. In this case, we also have $a\geq 1$.

Suppose that $b=0$. Since $\mathrm{char} (K)\neq 2 $, we have the following equalities by Equation~\eqref{eq:hgf}
$$
		2g-2 = N \bigg(-2 + \frac{1}{2} + \frac{3}{4} + \frac{2^a-1}{2^a}\bigg) 
		= N \bigg(\frac{1}{4} - \frac{1}{2^a}\bigg).
$$
Then the fact that $g\geq 2$ implies that $a \geq 3$; hence, $N \leq 16(g-1)$. 

\item[(c)] $e_2=3$:\\
Then $F $ is of type $(2,3,e_3)$, where $e_3 = 2\cdot3^b\ell^c$ for a prime $\ell > 3$ and integers $b,c \geq 0$.

\noindent If $c > 0 $, then $\ell=p$ and $e_3=6p^c$ by Lemma~\ref{lemma2} and Lemma~\ref{lem:two}$-(ii)$.
 Then by Lemma~\ref{1wild}, i.e., $d_3\geq (6p^c-1) +6(p^c-1)$, and by Equation~\eqref{eq:hgf}, we conclude that $N < 3(g-1)$.

\noindent Suppose that $c=0$; hence, $b\neq 0$. If $\mathrm{char}(K) \neq 3$, then $e_3=6$. By Equation~\eqref{eq:hgf}, we conclude that $g=1$, which is a contradiction. Therefore, $\mathrm{char}(K)= 3$, i.e., $P_2, P_3$ are wildly ramified. Then $N \leq 2(g-1)$ by Equation~\eqref{eq:hgf} and Lemma~\ref{1wild}. 
		
\item[(d)] $e_2=2$:\\
Then $F $ is of type $(2,2,e_3)$, where $e_3=2^a\ell^b$ for a prime $ \ell > 2 $ and integers $a,b \geq 0$.

Suppose that $b=0$. Since $\mathrm{char} (K) \neq 2$ , $F/F_0$ is a tame extension. Then by Equation~\eqref{eq:hgf}, we conclude that $g=0$, which is a contradiction. Therefore, $b> 0$. Then $\ell=\mathrm{char} (K)$, i.e., $\ell=p $. Hence, by Lemma~\ref{1wild} and Equation~\eqref{eq:hgf}, we have the following.
		\begin{align*}
		2g-2 &= N \bigg(-2 + 2\cdot \frac{1}{2} + \frac{(2^ap^b-1)+2^a(p^b-1 )}{2^ap^b}\bigg) \\
		&= N\frac{2^ap^b - 2^a -1}{2^ap^b} \geq N \bigg(1- \frac{2^a+1}{3 \cdot2^a} \bigg) \geq \frac{N}{3}.
		\end{align*}
		Hence, $N \leq 6(g-1)$. 
\end{itemize}

\end{itemize} 
\end{proof}

\begin{remark}
Note that if the bound $16(g-1)$ is attained by $F$, then $g-1$ is a power of $2$, $F$ is of type $(2,4,2^s)$ for some integer $s\geq 3$, and $ \mathrm{char}(K) \neq 2 $.
\end{remark}

We continue investigating $F$ of type $(e_1,e_2)$. 
That is, there are exactly $2$ ramified places of $ F_0 $, say $ P_1,P_2 $, with ramification indices $ e_1,e_2 $  and different exponents $ d_1,d_2$, respectively. Then we have the following result.

\begin{theorem}\label{thm:r=2}
 If $r=2$, then $N\leq 10(g-1)$.
\end{theorem}
\begin{proof}
We first remark that $F/F_0$ cannot be a tame extension; otherwise we would have that $g=0$. Therefore, we can write $N = p^t N_1$ for some positive integers  $t$ and $N_1 $ such that $\gcd(p,N_1)=1$. Note that $e_1=p^{a}N_1$ and $e_2=p^{b}N_1$ for some integers $0 \leq a,b\leq t$ by Lemma~\ref{lem:two}$-(ii)$. 
\begin{itemize}
\item[(i)] $N_1=1$, i.e.,  $G$ is a $p$-group:\\
Note that if $p^{a}=p^{b}=2$, then the case $d_1=d_2=2$ cannot hold; otherwise we would have that $g=1$. That is, $d_i\geq 3$ for some $i=1,2$, and hence $N\leq 4(g-1)$ by Equation~\eqref{eq:hgf}. 

\noindent Suppose that $p^{a}=p^{b}=2$ is not the case. Since $P_1$ and $P_2$ are ramified with different exponents $d_1\geq 2(p^{a}-1)$ and $d_2\geq 2(p^{b}-1)$, respectively, by Equation~\eqref{eq:hgf}
\begin{align}\label{eq:a,b}
2g-2 \geq N\left( -2+ \frac{2(p^{a}-1)}{p^{a}}+ \frac{2(p^{b}-1)}{p^{b}}\right) 
= N\left(2 -\frac{2}{p^{a}}-\frac{2}{p^{b}} \right)\geq \frac{N}{2}.
\end{align}
Therefore,  $N\leq 4(g-1)$.
\item[(ii)] $N_1>1$:
\begin{itemize}
\item[(a)] Suppose that $ F$ is of the type $ (N_1, N_1 p^b) $. Then $N=N_1 p^b$ by Lemmas~\ref{lem:two} and \ref{1wild}. Note that if $N< 10$, then $N< 10(g-1)$ as $g\geq 2$. Therefore, we suppose that $N \geq 10$. Moreover, by 
Lemma~\ref{1wild} and Equation~\eqref{eq:hgf}
\begin{align}\label{eq:pb}
	2g-2 
	&\geq N \bigg(-2 + \frac{N_1-1}{N_1} + \frac{(N_1p^b-1) +N_1(p^b-1) }{N_1p^b}\bigg) \\ \nonumber
	&= N\bigg(1-\frac{1}{N_1}-\frac{1}{N_1p^b}-\frac{1}{p^b}\bigg).
	\end{align}
Note that if $p^b\geq 5$, then $N\leq 10(g-1)$ by Equation~\eqref{eq:pb}. If $p^b = 4 $ and $N_1\geq 3$ or $p^b = 3$ and $N_1\geq 4$, then $N\leq 6(g-1)$. In the case that $p^b = 2 $ and $N_1\geq 5$, we obtain that $N\leq 10(g-1)$.
\item[(b)] 
Suppose that $ F $ is of type $ (N_1 p^a, N_1 p^b) $, where $0<a\leq b$. Then $d(P_1)\geq  (N_1p^a-1) +N_1(p^a-1) $ and $d(P_2)\geq  (N_1p^b-1) +N_1(p^b-1) $ by Lemma~\ref{1wild}. By Equation~\eqref{eq:hgf}, we obtain that
\begin{align*}
2g-2 
	&\geq N \bigg(2 - \frac{1}{N_1p^a} - \frac{1}{N_1p^b}  - \frac{1}{p^a} - \frac{1}{p^b}\bigg) 
	\geq N \bigg(2 - \frac{2}{N_1p^a}  - \frac{2}{p^a}\bigg) ,
\end{align*}
which implies that $N\leq 3(g-1)$.
\end{itemize}

\end{itemize}
\end{proof}

\begin{remark}\label{rem:2}
The bound $10(g-1)$ is attained only by $F$ of genus $2$ such that $F$ is of type $(5,10)$ if $\mathrm{char}(K)=2$ or $(2,10)$ if $\mathrm{char}(K)=5$.
\end{remark}

\begin{remark}
In \cite[Theorem 3.1]{KM}, the authors proved independently that if $ G $ is a $\ell$-subgroup of $ \mathrm{Aut}(F/K)$, where $\ell\geq 3$ is a prime and $\ell\neq \mathrm{char}(K)$, then $ |G| \leq 9(g - 1) $. They also showed that the equality can only be obtained for $ \ell = 3 $. 
Our analysis of the types of function fields with nilpotent automorphism groups not only leads us the same result, but also provides a bound for all nilpotent subgroups of $\mathrm{Aut}(F/K)$.  
\end{remark}

It remains to consider the case $ r=1 $. 
\begin{theorem}\label{thm:r=1}
	If $ r=1 $, then $\displaystyle N \leq \frac{4p}{(p-1)^2}g^2 $.
\end{theorem}
\begin{proof}
In Lemma~\ref{1wild}-$(ii)$, we observe that  
$ G $ is a $ p $-group and the unique ramified place $ P $ of $F_0$ is totally ramified in $F$. Therefore, the first ramification group $ G_1 $ of $ P $ is the whole group $G $. By \cite[Satz~1~(c)]{Sti1}, 
$$
|G_1| \leq \frac{4|G_2|}{(|G_2|-1)^2}g^2 \leq \frac{4p}{(p-1)^2}g^2, 
 $$ where $ G_2 $ is the second ramification group of $ P $. This gives the desired result. 
\end{proof}

\section{Examples}
In this section, we present examples of function fields that attain the bounds we obtained in Theorem~\ref{thm:r=4}, \ref{thm:r=3}, \ref{thm:r=2}, and \ref{thm:r=1}. In other words, the bounds given in these theorems cannot be improved. Moreover, for Theorem~\ref{thm:r=4} and Theorem~\ref{thm:r=3}, we construct a sequence of function fields $ F_n/K $ such that
\begin{itemize}
	\item [(i)] $ g(F_n) \rightarrow \infty $,
	\item [(ii)] there exists a nilpotent subgroup $ G_n \leq \mathrm{Aut}(F_n/K)$, whose order attains the respective bound. 
\end{itemize}

We need the following lemma to construct examples attaining the bound in Theorem~\ref{thm:r=3}.
\begin{lemma}\label{unramifiedext}
	Let $ \mathrm{char}(K) \neq 2$ and $ F_1/K $ be a Galois extension of $ F_0/K $ with $ g(F_1)\geq 2 $. 
	Then there exists a sequence of function fields $ F_n/K $ with the following properties.
	\begin{itemize}
		\item[(i)] $ F_n/F_0 $ is Galois,
		\item[(ii)] $ F_{n+1}/F_n $ is Galois, abelian of degree  $[F_{n+1}:F_n] = 2^{2g(F_n)} $,
		\item[(iii)] $ F_{n+1}/F_n $ is unramified, 
		\item[(iv)] the exponent of $ \mathrm{Gal}(F_{n+1}/F_n) $ is $ 2 $.
	\end{itemize}
\end{lemma}
\begin{proof}
By \cite[Section~4.7]{PS}, for a given function field $ F/K $, there exists a unique maximal field $ F' \supseteq F $ such that 
\begin{itemize}
	\item[(a)] $ F'/F $ is Galois and abelian of degree  $[F':F] = 2^{2g(F)} $,
	\item[(b)] $ F'/F $ is unramified, and 
	\item[(c)] the exponent of $ \mathrm{Gal}(F'/F) $ is $ 2 $.
\end{itemize}
For $ n \geq 1 $, let $ F_{n+1} $ be the extension of $ F_n $ described as above. Now we show that $ F_n/F_0 $ is a Galois extension for each $ n\geq 1 $. We proceed by induction on $ n $. By our assumption, $ F_1/F_0 $ is Galois. Now suppose that $ F_n/F_0 $ is Galois.  
Let $ \tilde{F} $ be the Galois closure of $ F_{n+1}/F_0 $, see Figure~\ref{4}.
\begin{figure}[!ht]
	\begin{center}{
			\xymatrix{
				&&&&&&&&\tilde{F}&&\\
				&&&&&&&&F_{n+1} \ar@{-}[u] &&\\
				&&&&&&&&F_n\ar@{-}[u] \ar@{-}@/_1pc/[u]_{\text{Galois } }&&\\
				&&&&&&&&F_0 \ar@{-}[u]\ar@{-}@/_1pc/[u]_{\text{Galois}  }\ar@{-}@/^2pc/[uuu]^{\text{Galois}}&& 
		}}
	\end{center}
	\caption{The Galois closure of $F_{n+1}/F_0$} 
	\label{4}
\end{figure}
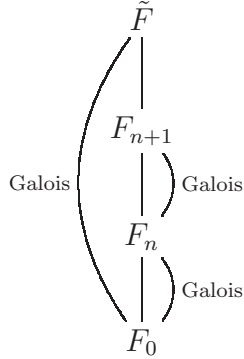	
Let $ \gamma \in \mathrm{Gal}(\tilde{F}/F_0)$. Since $ F_n/F_0 $ is Galois, we have $ \gamma(F_n) = F_n $. In particular, $ \gamma(F_{n+1}) $ is a Galois, abelian, unramified extension of $ F_n $ of degree power of $ 2$. By the uniqueness of such an extension, we conclude that $ \gamma(F_{n+1})=F_{n+1} $, which gives the desired result.
\end{proof}

The following example shows that the bound in Theorem~\ref{thm:r=3} is attained by function fields of infinitely many genera. We apply a similar approach as in \cite{KM}.
%
\begin{example}\label{ex:r=3}
	Let $ p \neq 2 $ and $ \zeta $ be a primitive $8$-th root of unity.
	Consider the function field $ F $ with the defining equation 
	$$  
	y^2 = x(x^4-1). 
	$$
Note that $F/K(x)$ is a Kummer extension of degree $2$, where $(x=\infty)$, $(x=0)$ and $(x=\zeta^{2k})$, $k=0,1,2,3$, are the ramified places of $K(x)$, see \cite[Proposition 3.7.3]{thebook}. Hence we conclude that $ g(F)=2 $. Note that the maps 
	\begin{align*}
	\sigma :\left\{
	\begin{array}{lr} 
	x \mapsto \zeta^2 x \\ 
	y \mapsto \zeta y \\ 
	\end{array}
	\right.
	\qquad
	\text{ and }
	\qquad
	\tau:\left\{
	\begin{array}{lr}
	x \mapsto -1/x \\
	y \mapsto y/x^3
	\end{array}
	\right.
	\vspace*{-.5cm}
	\end{align*}
	define automorphisms of $ F$ over $K$. Set $\eta :=\sigma \tau$. Then $\sigma, \eta \in \mathrm{Aut}(F/K)$ such that $ \mathrm{ord}(\sigma) = 8$, $ \mathrm{ord} (\eta) = 2 $ and $\eta\sigma \eta^{-1}=\sigma^3$. Let $G=\langle \sigma \rangle \rtimes  \langle\eta \rangle$. Then $G$ is a group of order $16$. Set $ z:= x^4 $ and $ t:= \displaystyle \frac{z^2+1}{2z} $. We consider the sequence of function fields $K(t) \subseteq K(z) \subseteq K(x) \subseteq F$ to investigate the ramification structure in $F/K(t)$, see Figure~\ref{5}.

\begin{figure}[!ht]
		\begin{center}{
				\xymatrix{
					&&&&&&&{F}=K(x,y)&&\\
					&&&&&&& K(x) \ar@{-}[u]^{y^2 = x(x^4-1) \;\;}_{\;\;\mathrm{deg}=2}&&\\
					&&&&&&&K(z)\ar@{-}[u]^{z=x^4 \;\;}_{\;\; \mathrm{deg}=4}&&\\
					&&&&&&&K(t) \ar@{-}[u]^{t= \frac{z^2+1}{2z} \;\;}_{\;\;\mathrm{deg}=2}&& }}
		\end{center}
		\caption{$K(t) \subseteq K(z) \subseteq K(x) \subseteq F$ } \label{5}
	\end{figure}

\noindent Observe that $ K(t) \subseteq F^{\langle \sigma \rangle} $ and $ K(t) \subseteq F^{\langle \tau \rangle} $; hence, $ K(t) \subseteq F^G $. Then the fact that $[F:K(t)]~=~16$ implies that $F^G=K(t)$, that is, $ F/K(t) $ is a Galois extension of degree $ 16 $.  
Then we have the following observations.
\begin{itemize}
\item[(i)] $(z=0)$ and $(z=\infty)$ are the only ramified places of $K(z)$ in $K(x)$, which are totally ramified. 
$(x=0)$ and $(x=\infty)$ are the unique places lying over $(z=0)$ and $(z=\infty)$, respectively. That is, $(z=0)$ and $(z=\infty)$ are totally ramified in $F$. Furthermore, $(x=\zeta^{2k})$, $k=0,1,2,3$ are the places lying over $(z=1)$.
\item[(ii)] Ramified places of $K(t)$ in $K(z)$ are $(t=1)$ and $(t=-1)$ lying over $(z=1)$ and $(z=-1)$, respectively. Furthermore, $(z=0)$ and $(z=\infty)$ lie over $(t=\infty)$. 
\end{itemize}	
Hence, we conclude that the ramified places of $ K(t) $ in $ F $ are $(t=-1) $, $(t=1) $ and $(t=\infty) $ with ramification indices $ 2,4,8$, respectively. That is, $ F $ is of type $ (2,4,8)$ and $N=16=16(g(F)-1)$.

Set $ F_0 = K(t) $ and $ F_1 := F $. Then by Lemma~\ref{unramifiedext}, there exists a sequence of function fields $ F_n/K $ such that $F_{n+1}/K(t) $ is a Galois extension of degree power of $ 2 $. Note that $ g(F_{n+1}) > g(F_{n})$ as $ g(F_1)=2 $. Since $ F_{n+1}/F_1 $ is an unramified extension, $ K(t) $ has  exactly $3$ ramified places in $ F_{n+1} $, namely, $ (t=-1) $, $ (t=1) $ and $ (t=\infty) $, whose ramification indices are $2$, $4$, $8$, respectively. Thus, $ F_{n+1} $ is of type $ (2,4,8) $ such that $ [F_{n+1}:K(t)] = 16 (g(F_{n+1})-1)  $.
	
\end{example}

By using Example~\ref{ex:r=3}, we obtain the following example, which shows that  the bound given in Theorem~\ref{thm:r=4} is attained by function fields of infinitely many genera.
\begin{example} \label{ex:r=4}
Let $F_n/K(t)$ be the Galois extension given in Example \ref{ex:r=3} for $n\geq 1$. Recall that $ p \neq 2 $. We consider the Kummer extension $K(w)/K(t)$ given by $w^2=t-1$. Note that $w\in F_n$ as $t$ is a square in $F_n$, i.e., $F_n/K(w)$ is a Galois extension of degree power of $ 2$, see Figure~\ref{6}.
\begin{figure}[!ht]
		\begin{center}{
				\xymatrix{
					&&&&&&&{F_n}&&\\
					&&&&&&& &&\\
					&&&&&&& & K(w)\ar@{-}[uul]_{\;\;\mathrm{deg}= \frac{[F_n:K(t)]}{2}}&\\
					&&&&&&&K(t)\ar@{-}[uuu]^{\;\;}\ar@{-}[ur]_{w^2=t-1}&& }}
		\end{center}
		\caption{$K(t) \subseteq K(w) \subseteq F_n$} \label{6}
	\end{figure}
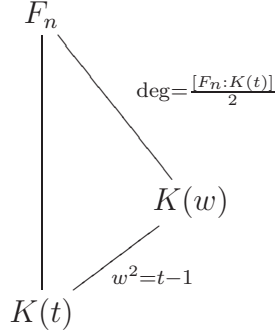	
By \cite[Proposition 3.7.3]{thebook} , $(t=1)$ and $(t=\infty)$ are the only ramified places of $K(t)$ in $K(w)/K(t)$. In particular, $(w=0)$ and $(w=\infty)$ are the places of $K(w)$ lying over $(t=1)$ and $(t=-1)$, respectively. Moreover, $(w=\alpha)$ and $(w=-\alpha)$ are the places of $K(w)$ lying over $(t=-1)$, where $\alpha^2=-2$. Then the transitivity of the ramification extension, we conclude that $(w=\alpha)$, $(w=- \alpha)$, $(w=0)$ and $(w=\infty)$ are the only ramified places of $K(w)$ in $F_n$ and the ramification indices are given by 
\begin{align*}
e((w= \alpha))=e((w=- \alpha))=e((w=0))=2 \quad \text{and} \quad e((w=\infty))=4 . 
\end{align*}
Hence, $F_n$ is a function field of type $(2,2,2,4)$ satisfying $[F_n:K(w)]=8(g(F_n)-1)$.
\end{example}

The following two examples show that both cases, where the bound in Theorem~\ref{thm:r=2} can be attained, appear, see Remark \ref{rem:2}.
\begin{example} \label{ex:r=2}
Let $ p=2 $ and let $ F $ be a function field given by the defining equation $y^2-y =x^5$. By considering $ F $ as a Kummer extension over $ K(y) $, where $ (y=\infty) $, $ (y = 0) $ and $ (y=1) $ are all the ramified places of $K(y)$, we conclude that $ g(F)=2 $ by \cite[Proposition 3.7.3]{thebook}.
 
\noindent Set $ z:= x^5 $. Then $K(x)/K(z)$ and $K(y)/K(z)$ are Galois extensions of degree $5$ and $2$, respectively. Hence, $F/K(z)$ is a Galois extension of degree $10$, see Figure~\ref{3}. Note that the automorphism group of $ F/K(z) $ is generated by $\sigma$ defined by
	$$
	\sigma: \left\{
	\begin{array}{lr} 
	x \mapsto \zeta x \\ 
	y \mapsto y + 1 ,\\ 
	\end{array}
	\right.
	$$
	where $ \zeta $ is a primitive $5$-root of unity. 
	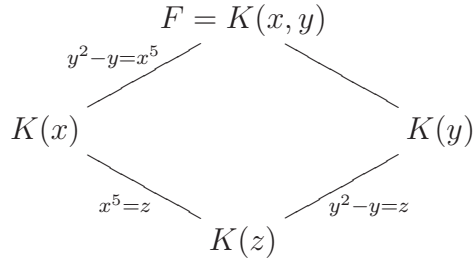
\begin{figure}[!ht]
		\begin{center}{
				\xymatrix{
					&&&&&& F =K(x,y)&&&&&\\
					&&&&&K(x) \ar@{-}[ur]^{y^2-y=x^5}&&K(y) \ar@{-}[ul]&&&& \\
					&&&&&& K(z) \ar@{-}[ur]_{y^2-y=z}\ar@{-}[ul]^{x^5=z}&&&&&}}
		\end{center}
		\caption{$F$ as a compositum of $K(x)$ and $K(y)$} \label{3}
	\end{figure}

\noindent Note that $(z=\infty)$ is the only ramified place in $K(y)$ with ramification index $2$. Also, $(z=0)$ and $(z=\infty)$ are the only ramified places in $K(x)$ with ramification indices $5$. Then, by Abhyankar's Lemma \cite[Theorem 3.9.1]{thebook}, $(z=0)$ and $(z=\infty)$ are the only ramified places of $K(z)$ in $F$ with the ramification indices $5$ and $10$, respectively. That is, $ F $ is of type $ (5,10) $ satisfying $ [F:K(z)] = 10 (g(F)-1)$.

\vspace{0.2cm}
	
\noindent Let $ p=5 $ and let $ F $ be a function field given by the defining equation $y^5-y =x^2$. Similarly, $F$ is a function field of genus $2$. If we set $ z:= x^2 $, then $F/K(z)$ is a cyclic extension of degree $10$, where $ (z=0) $ and $ (z=\infty) $ are all the ramified places of $K(z)$ in $F$ with ramification indices $2$ and $10$, respectively. That is, $ F $ is of type $ (2,10) $ satisfying $ [F:K(z)] = 10(g(F)-1)$.  Note that the automorphism group of $ F/K(z) $  is generated by $\sigma$ defined by
	$$
	\sigma: \left\{
	\begin{array}{lr} 
	x \mapsto \zeta x \\ 
	y \mapsto y + 1 ,\\ 
	\end{array}
	\right.
	$$
where $ \zeta $ is a primitive $2$-root of unity.

\end{example}

The following example shows that the bound in Theorem \ref{thm:r=1} holds, for further details see \cite[Satz~5]{Sti2}.

\begin{example}\label{ex:r=1}
	 Let $p\geq 5 $ and $n \geq 1 $ be integers. 
	Consider the function field $ F_n:= K(x,y) $ defined by 
	$$
	y^{p}+ y =x^{p^{n}+1}.
	$$
Then $g(F_n)= \frac{p^{n}(p-1)}{2} $.
	Note that the pole divisors of $ x,y $ are $ (x)_\infty = p\cdot P $ and $ (y)_\infty = (p^n+1)\cdot P $, respectively, for a place $ P $ of $ F_n$. 
	Let $ G=(\mathrm{Aut}(F_n/K))_{P} $ be the  automorphism group fixing the unique pole $ P $ of $ x $ and $ y $. The group $ G $ consists of automorphisms of the form 
	$$ \sigma:\left\{
	\begin{array}{lr}
	x \mapsto x+d, \\
	y \mapsto y+Q(x),
	\end{array}
	\right.
	$$ 
	where $ p\deg Q(x) \leq p^n $ and $ Q(x)^p + Q(x) = (x+d)^{p^n+1} -x^{p^n+1} $. In this case, $|G| = p^{2n+1}$ and $\displaystyle |G| = \frac{4p}{(p-1)}g^2 $.
\end{example}

	\section*{Acknowledgements}

	The authors would like to thank Prof.~Dr.~Henning Stichtenoth and Dr.~Maria Montanucci for the helpful discussions and comments, which improved the manuscript considerably.
	N. A. is supported by B.A.CF-19-01967 and B.G. is supported by T\"{U}B\.{I}TAK--2214A Fellowship Program.

\end{document}